\documentclass{amsart}

\usepackage{amssymb, amsmath, amsthm, stmaryrd, amscd, mathrsfs}
\SetSymbolFont{stmry}{bold}{U}{stmry}{m}{n} % Gegen Warnungen bei stmry

\usepackage{verbatim}
\usepackage{dsfont}
\usepackage{enumerate}

\allowdisplaybreaks

\newtheorem{Def}{Definition}[section]
\newtheorem{Lem}[Def]{Lemma}
\newtheorem{Pro}[Def]{Proposition}

\newtheorem{The}[Def]{Theorem}

\newtheorem{Rem}[Def]{Remark}

\numberwithin{equation}{section}

\renewcommand{\hat}{\widehat}

\begin{document}

\author{Gieri Simonett}
\address{\!\!\!Department of Mathematics, Vanderbilt University, Nashville, TN, USA}
\email{gieri.simonett@vanderbilt.edu}

\author{Mathias Wilke}
\address{\!\!\!Faculty of Mathematics, University of Regensburg, Germany}
\email{mathias.wilke@ur.de}

%\subjclass[2010]{14}

\dedicatory{Dedicated to Jan Pr\"uss on the occasion of his retirement}

\thanks{The research of the first author was partially supported by NSF DMS-1265579.}

%%%%%%%%%%%%%%%%%%%%%%%%%%%%%%%%%%%%%%%%%%%%%%%%%%%%%%%%%%%%%%%%%%%%%%%%%%%%%%%%%%%%%%%%%%%%%%%%%%%%%%%%%%%%%%%%%%%%%%%%%%%%%%%%%%%%%%%%%%%%%%%%

\begin{abstract}
We investigate the Westervelt equation from nonlinear acoustics, subject to nonlinear absorbing boundary conditions of order zero, which were recently proposed in \cite{KaSh15,KaSh15.1}. We apply the concept of maximal regularity of type $L_p$ to prove global well-posedness for small initial data. Moreover, we show that the solutions regularize instantaneously which means that they are $C^\infty$ with respect to time $t$ as soon as $t>0$. Finally, we show that each equilibrium is stable and each solution which starts sufficiently close to an equilibrium converges at an exponential rate to a possibly different equilibrium.
\end{abstract}

\title[Westervelt equation with absorbing boundary conditions]{Well-posedness and long-time behavior for the Westervelt equation with absorbing boundary conditions of order zero}

%\keywords{Reaction-diffusion systems, multicomponent reactive mixtures, Maxwell-Stefan diffusion, reversible mass-action kinetics, maximal $L_p$-regularity, generalized principle of linearized stability, free energy, convergence to equilibria.}
%\subjclass[2000]{35R35, Secondary: 35Q30, 76D45, 76T10.}

\maketitle

\section{Introduction and the model}

We are concerned with the so-called Westervelt equation
\begin{align}\label{eq:Westervelt0}
  u_{tt} - c^2\Delta u - \beta \Delta u_t &= \gamma (u^2)_{tt},
\end{align}
which is used to describe the propagation of sound in fluidic media.  The function $u(t,x)$ denotes the acoustic pressure fluctuation from an ambient value at time $t$ and position $x$.  Furthermore, $c>0$ denotes the velocity of sound, $\beta>0$ the diffusivity of sound and $\gamma>0$ the parameter of nonlinearity.  The Westervelt equation can be regarded as a simplification of Kuznetsov's equation
\begin{align}\label{eq:Kuznetsov}
 u_{tt} - c^2\Delta u- \beta \Delta u_t &= \gamma (u^2)_{tt}+|v|^2_{tt}.
\end{align}
Here the velocity fluctuation $v(t,x)$ is related to the pressure fluctuation by means of an acoustic potential $\psi(t,x)$, such that $u=\rho_0\psi_t$, $v=-\nabla\psi$ with ambient density $\rho_0>0$. This equation is used as a basic equation in nonlinear acoustics, see \cite{HaBl98,Kal07,LeSeWo08}. It can be derived from the balances of mass and momentum (the compressible Navier-Stokes equations for Newtonian fluids) and a state equation for the pressure-dependent density of the fluid.  We refer to \cite{Kal07} for a derivation of Kuznetsov's equation.

Observe that the left hand side of \eqref{eq:Westervelt0} is a strongly damped wave equation, which is of parabolic type. Taking the highest order terms on the right hand side of \eqref{eq:Westervelt0} into account, we claim that parabolicity is preserved provided that the pressure fluctuation $u$ admits values which are sufficiently close to zero. To see this, we use the identity
$$(u^2)_{tt}=2uu_{tt}+2(u_t)^2,$$
wherefore we may rewrite \eqref{eq:Westervelt0} as follows:
$$(c^{-2}-2\gamma u)u_{tt}- \Delta u - \beta \Delta u_t = 2\gamma (u_t)^2.$$
Consequently we see that \eqref{eq:Westervelt0} degenerates as $u$ gets close to $\frac{1}{2\gamma c^2}$. To this end we allow the function $|u|$ to take values in the interval $[0,\frac{1}{2\gamma c^2})$ in order to use features from the parabolic theory for PDEs.

If one considers the Westervelt equation \eqref{eq:Westervelt0} in a bounded framework, i.e. $x\in \Omega$ and $\Omega\subset\mathbb{R}^d$ is open and bounded, then one has to equip \eqref{eq:Westervelt0} with suitable boundary conditions on the boundary $\partial\Omega$. The Westervelt (resp.\ Kuznetsov) equation with linear boundary conditions of Dirichlet- or Neumann-type has been analyzed by a number of authors, see e.g.\ \cite{ClKaVe09,Kal10,KaLa09, KaLa11,KaLa12,MeWi11,MeWi13}, which is just a selection. The basic difference is the choice of the functional analytic setting. While in \cite{ClKaVe09,Kal10,KaLa09,KaLa11,KaLa12} the analysis is based on  $L_2$-theory and energy estimates, the authors in \cite{MeWi11,MeWi13} use the technique of maximal regularity of type $L_p$ and obtain optimal regularity results, which is feasible by the parabolic nature of \eqref{eq:Westervelt0} or \eqref{eq:Kuznetsov} as long as $u$ is close to zero. Moreover, in \cite{KaLa09,KaLa12,MeWi11,MeWi13}, the authors prove exponential stability of the trivial solution $u=0$ of the Westervelt or Kuznetsov equation with homogeneous Dirichlet boundary conditions.

From a point of view of applications one is often confronted with the situation that the region of interest is small compared to the underlying acoustic propagation domain. One way out of this problem is to truncate the large domain and to equip \eqref{eq:Westervelt0} or \eqref{eq:Kuznetsov} with so-called \emph{absorbing boundary conditions}. Recently, Kaltenbacher \& Shevchenko \cite{KaSh15,KaSh15.1} derived and proposed absorbing boundary conditions of order zero and order one for the Westervelt equation \eqref{eq:Westervelt0} in one and two space dimensions. This type of boundary conditions can e.g.\ be interpreted as a kind of feedback control for stabilizing \eqref{eq:Westervelt0}. In this paper we consider absorbing boundary conditions of order zero, which look as follows:
\begin{equation}\label{eq:ABCzero}
\partial_\nu (u+\beta u_t)+u_t\sqrt{c^{-2}-2\gamma u}=0\quad\text{on}\ \partial\Omega.
\end{equation}
Here $\nu$ is the outer unit normal vector field on $\partial\Omega$ and $\partial_\nu$ denotes the normal derivative. At this point we want to emphasize that in contrast to the classical Dirichlet- or Neumann boundary conditions, the boundary condition \eqref{eq:ABCzero} is nonlinear.

Complementing \eqref{eq:Westervelt0} with initial conditions for $u$ and $u_t$, we end up with the initial boundary value problem
\begin{align}\label{eq:westervelt}
\begin{split}
c^{-2}u_{tt} - \Delta u - \beta \Delta u_t &= \gamma (u^2)_{tt},\quad\text{in }J\times\Omega,\\
      \partial_\nu (u+\beta u_t)+u_t\sqrt{c^{-2}-2\gamma u}&= 0,\quad\text{in }J\times\partial\Omega,\\
      (u(0),u_t(0)) &= (u_0,u_1) ,\quad\text{in }\Omega,
     \end{split}
\end{align}
for the Westervelt equation, where $J=(0,T)$ for some $T\in (0,\infty)$, $\Omega\subset\mathbb{R}^d$, $d\in\mathbb{N}$, is a bounded domain with boundary $\partial\Omega\in C^2$ and the parameters $c>0$, $\beta>0$ and $\gamma>0$ are given.

To the best of the authors' knowledge there seems to be only the article by Kaltenbacher \& Shevchenko \cite{KaSh15} which deals with the analysis of problem \eqref{eq:westervelt} in one and two space dimensions (the proofs of the results in \cite{KaSh15} are carried out in \cite{KaShArx}). The technique used in \cite{KaShArx,KaSh15} to establish well-posedness is based on an $L_2$-theory and energy estimates combined with the contraction mapping principle. The article \cite{KaSh15} is complemented with some numerical results, showing that the absorbing boundary conditions proposed and derived in \cite{KaSh15} demonstrate more accurate numerical results as compared to those proposed by Engquist \& Majda \cite{EngMa79}. Finally, it should be noted that the authors in \cite{KaSh15} also derive absorbing boundary conditions of first order which for the Westervelt equation result in dynamic boundary conditions for the pressure fluctuation $u$.

The present paper provides a rather complete analysis of problem \eqref{eq:westervelt}. We will present optimal conditions on the initial data $(u_0,u_1)$ for the existence and uniqueness of a solution to \eqref{eq:westervelt}, thereby improving the assumptions on $(u_0,u_1)$ in \cite{KaSh15} (for details see below). In addition, we investigate the temporal regularity of the solutions to \eqref{eq:westervelt} as well as their long-time behavior.

Our program for studying \eqref{eq:westervelt} is as follows. In Section 2 we consider the principal linearization of \eqref{eq:westervelt} in $u=0$ and we prove optimal regularity results of type $L_p$ for the resulting parabolic problem. Unfortunately one cannot directly apply the results in \cite{DHP03,DHP07} or \cite{LPS06} to the linear problem, since after a transformation of \eqref{eq:westervelt} to a first order system with respect to the variable $t$, the principal linearization is neither parameter elliptic nor normally elliptic. Instead we will treat the linearization of \eqref{eq:westervelt} in its original second order formulation as it has already been done in \cite{MeWi13} for the Kuznetsov equation \eqref{eq:Kuznetsov} with Dirichlet boundary conditions.

Section 3 is devoted to the proof of the following result concerning well-posedness of \eqref{eq:westervelt} under optimal conditions on the initial value $(u_0,u_1)$.
\begin{The}\label{thm:mainthm}
Let $p>d+1$, $p\neq 3$, $J=(0,T)$, $\Omega\subset\mathbb{R}^d$, $d\in\mathbb{N}$ be a bounded domain with boundary $\partial\Omega\in C^2$. Then for each $T\in(0,\infty)$ there exists $\delta>0$ such that for all $(u_0,u_1)\in W_p^2(\Omega)\times W_p^{2-2/p}(\Omega)=:X_\gamma$ satisfying
the estimate
$$\|u_0\|_{W_p^2(\Omega)}+\|u_1\|_{W_p^{2-2/p}(\Omega)}\le\delta,$$  and the compatibility condition
\begin{equation}\label{eq:compcondthm}
\partial_\nu (u_0+\beta u_1)+u_1\sqrt{c^{-2}-2\gamma u_0}= 0,\quad\text{on }\partial\Omega
\end{equation}
if $p>3$,
there is a unique solution
$$u\in W_p^2(J;L_p(\Omega))\cap W_p^1(J;W_p^2(\Omega))=:\mathbb{E}_1(J)$$
of \eqref{eq:westervelt}. In addition, the solution satisfies
$$\|u\|_\infty:=\max_{(t,x)\in [0,T]\times\overline{\Omega}}|u(t,x)|<\frac{1}{2\gamma c^2}$$
and the data-to-solution map
$$[(u_0,u_1)\mapsto u(u_0,u_1)]:B_{X_\gamma}(0,\delta)\to \mathbb{E}_1(J)$$
is continuous.
\end{The}
For the proof of Theorem \ref{thm:mainthm} we employ the implicit function theorem and the results on optimal regularity of the linearization (see Section 2) in a neighborhood of $u=0$. This in turn yields the desired bound $\|u\|_\infty<\frac{1}{2\gamma c^2}$. At this point we want to emphasize that the assertions of Theorem \ref{thm:mainthm} remain true if one replaces the assumption $p>d+1$ by the weaker condition $p>\max\{\frac{d}{2},\frac{d}{4}+1\}$ (cf.\ Remark \ref{rem:remarkonp}). Therefore the case $p=2$ can be covered provided that $d\le 3$. In particular, for $p=2$, the initial value $(u_0,u_1)$ has to be small in $W_2^2(\Omega)\times W_2^1(\Omega)$ compared to the assumption in \cite[Theorems 3.1 \& 3.2]{KaSh15} where $(u_0,u_1)$ has to be small in $W_2^2(\Omega)\times W_2^2(\Omega)$. Thus, we were able to reduce the regularity for $u_1$. Note that the compatibility condition \eqref{eq:compcondthm} is not needed in case $p<3$.

In Section 4 we study the regularity of the solution with respect to the temporal variable $t$. We use a parameter trick which goes back to Angenent \cite{Ang90}, combined with the implicit function theorem to prove that the solution is infinitely many times differentiable with respect to $t$ as soon as $t>0$, see Theorem \ref{thm:HR}. This result reflects the parabolic regularization effect.

Finally, in Section 5, we address the question about the long-time behavior of solutions to \eqref{eq:westervelt}. For that purpose we reformulate \eqref{eq:westervelt} as a first order system with respect to $t$ and consider the set $\mathcal{E}$ of equilibria, given by
$$\mathcal{E}=\left\{(r,0): r\in\mathbb{R},\ |r|<\frac{1}{2\gamma c^2}\right\}.$$
If $A_0$ denotes the full linearization in $(r,0)\in\mathcal{E}$, we prove that
\begin{itemize}
\item $0\in\sigma(A_0)$ is semi-simple and
\item $\sigma(A_0)\backslash\{0\}\subset\mathbb{C}_-=\{z\in\mathbb{C}:\operatorname{Re}z<0\}$.
\end{itemize}
Relying on the maximal regularity results from Sections 2 \& 3 and applying the results in \cite{PSZ09} this implies that each $(r,0)\in \mathcal{E}$ is stable (in the sense of Lyapunov) and each solution of \eqref{eq:westervelt} with initial values sufficiently close to $(r,0)$ converges at an exponential rate to a possibly different equilibrium as $t\to\infty$, see Theorem \ref{thm:Stab}.

\section{Maximal regularity of the linearization}

Let's take a look at the regularity of $u$ at the boundary. For
$$u\in W_p^2(J;L_p(\Omega))\cap W_p^1(J;W_p^2(\Omega))$$
we have by trace theory
$$\partial_\nu u_t\in W_p^{1/2-1/2p}(J;L_p(\partial\Omega))\cap L_p(J;W_p^{1-1/p}(\partial\Omega)),$$
$$\partial_\nu u\in W_p^{3/2-1/2p}(J;L_p(\partial\Omega))\cap W_p^1(J;W_p^{1-1/p}(\partial\Omega)),$$
and
$$u_t|_{\partial\Omega}\in W_p^{1-1/2p}(J;L_p(\partial\Omega))\cap L_p(J;W_p^{2-1/p}(\partial\Omega)),$$
hence $\partial_\nu u$ as well as $u_t|_{\partial\Omega}$ carry additional time regularity compared to $\partial_\nu u_t$. The same holds for the term $\Delta u$ compared to $u_{tt}$ and $\Delta u_t$, since
$$\Delta u\in W_p^1(J;L_p(\Omega)).$$
We use these facts for the terms $\partial_\nu u$ as well as $\Delta u$ and study in a first step the linear problem
\begin{align}\label{eq:linwestervelt}
\begin{split}
c^{-2}u_{tt} - \beta \Delta u_t &= f,\quad\text{in }J\times\Omega,\\
      \beta \partial_\nu u_t+\alpha u_t&= g,\quad\text{in }J\times\partial\Omega,\\
      (u(0),u_t(0)) &= (u_0,u_1) ,\quad\text{in }\Omega,
     \end{split}
\end{align}
for $\alpha\ge 0$ and given functions $f\in L_p(J;L_p(\Omega))$,
$$g\in W_p^{1/2-1/2p}(J;L_p(\partial\Omega))\cap L_p(J;W_p^{1-1/p}(\partial\Omega))$$
$(u_0,u_1)\in W_p^2(\Omega)\times W_p^{2-2/p}(\Omega)$ satisifying the compatibility condition
\begin{equation}\label{eq:vertrbed}
\beta\partial_\nu u_1+\alpha u_1= g(0)
\end{equation}
on $\{t=0\}\times\partial\Omega$ if $p>3$.

Let us solve the problem
\begin{align}\label{eq:linwestervelt2}
\begin{split}
c^{-2}v_t - \beta \Delta v &= f,\quad\text{in }J\times\Omega,\\
      \beta \partial_\nu v+\alpha v&= g,\quad\text{in }J\times\partial\Omega,\\
      v(0) &= u_1 ,\quad\text{in }\Omega,
     \end{split}
\end{align}
by \cite[Theorem 2.1]{DHP07} to obtain a unique solution
$$v\in W_p^1(J;L_p(\Omega))\cap L_p(J;W_p^2(\Omega)).$$
This is possible since the given functions $(f,g,u_1)$ belong to the optimal regularity classes and the compatibility condition \eqref{eq:vertrbed} holds. Then we define
$$u(t,x):=u_0(x)+\int_0^t v(s,x)ds$$
for all $t\in [0,T]$ and $x\in\Omega$. Clearly we have $u(0,x)=u_0(x)$, $\partial_t^ju(t,x)=\partial_t^{j-1}v(t,x)$, $j=1,2$,
$$u\in L_p(0,T;W_p^2(\Omega)),\quad u_t=v\in W_p^1(J;L_p(\Omega))\cap L_p(J;W_p^2(\Omega))$$
and
$$u_{tt}=v_t\in L_p(J;L_p(\Omega))$$
for $J=[0,T]$ and every finite $T>0$.
This implies that
$$u\in W_p^2(J;L_p(\Omega))\cap W_p^1(J;W_p^2(\Omega))$$
and $u$ solves \eqref{eq:linwestervelt}, showing existence. To prove uniqueness, assume that $u_1,u_2$ solve \eqref{eq:linwestervelt}, hence $u:=u_1-u_2$ solves \eqref{eq:linwestervelt} with $(f,g,u_0,u_1)=0$. Defining $v:=u_t$ it follows that $v$ solves \eqref{eq:linwestervelt2} with trivial data $(f,g,u_1)=0$. Since the solution to \eqref{eq:linwestervelt2} is unique, it follows that $v=0$, hence $u_t=0$, hence $u=u_0=0$. We summarize the preceeding result in the following
\begin{Lem}\label{lem:linwestervelt1}
Let $p\in (1,\infty)$, $p\neq 3$, $\alpha\ge 0$, $J=(0,T)$, $T\in (0,\infty)$, $\Omega\subset\mathbb{R}^d$, $d\in\mathbb{N}$, be a bounded domain with boundary $\partial\Omega\in C^2$. Then there exists a unique solution
$$u\in W_p^2(J;L_p(\Omega))\cap W_p^1(J;W_p^2(\Omega))=:\mathbb{E}_1(J)$$
of \eqref{eq:linwestervelt} if and only if the data satisfy the following conditions.
\begin{enumerate}
\item $f\in L_p(J;L_p(\Omega))=:\mathbb{E}_0(J)$;
\item $g\in W_p^{1/2-1/2p}(J;L_p(\partial\Omega))\cap L_p(J;W_p^{1-1/p}(\partial\Omega))=:\mathbb{F}(J)$;
\item $(u_0,u_1)\in W_p^2(\Omega)\times W_p^{2-2/p}(\Omega)=:X_\gamma$;
\item $\beta\partial_\nu u_1+\alpha u_1=g(0)$ if $p>3$.
\end{enumerate}
\end{Lem}
\begin{proof}
It remains to prove the necessity of the conditions. If
$$u\in W_p^2(J;L_p(\Omega))\cap W_p^1(J;W_p^2(\Omega))$$
solves \eqref{eq:linwestervelt} then
$$u_t\in  W_p^1(J;L_p(\Omega))\cap L_p(J;W_p^2(\Omega))$$
solves \eqref{eq:linwestervelt2} and the assertions for $f,g$ and $u_1$ follow from the equations and trace theory, see e.g.\ \cite{DHP07,PS16}. Finally, by Sobolev embedding, we obtain
$$u\in W_p^1(J;W_p^2(\Omega))\hookrightarrow C([0,T];W_p^2(\Omega)),$$
hence $u_0\in W_p^2(\Omega)$.
\end{proof}
For $u\in \mathbb{E}_1(J)$ let
$$Lu:=[c^{-2}u_{tt}-\beta\Delta u_t,\beta\partial_\nu u_t+\alpha u_t,(u(0),u_t(0))].$$
With this notation it follows from Lemma \ref{lem:linwestervelt1} that the linear mapping
$$L:\mathbb{E}_1(J)\to \{(f,g,(u_0,u_1))\in \mathbb{E}_0(J)\times\mathbb{F}(J)\times X_\gamma:\beta\partial_\nu u_1+\alpha u_1=g(0)\ \text{if}\ p>3\}$$
is a bounded isomorphism with a bounded inverse. It will be convenient to introduce the following subspaces of $\mathbb{F}(J)$ and $\mathbb{E}_1(J)$. Let
$$_0\mathbb{F}(J):=\{g\in \mathbb{F}(J):g(0)=0\}$$
and
$$_0\mathbb{E}_1(J):=\{u\in \mathbb{E}_1(J):u(0)=u_t(0)=0\}.$$
For $u\in\, _0\mathbb{E}_1(J)$ let
$$L_0u:=[c^{-2}u_{tt}-\beta\Delta u_t,\beta\partial_\nu u_t+\alpha u_t].$$
Then, by Lemma \ref{lem:linwestervelt1}, the mapping
$$L_0:\, _0\mathbb{E}_1(J)\to \mathbb{E}_0(J)\times\, _0\mathbb{F}(J)$$
is a bounded isomorphism and by standard reflection arguments it can be shown that the norm of the inverse $L_0^{-1}$ is independent of $T\in (0,T_0]$ for every fixed $T_0>0$.

We will now take care of the lower order terms $\partial_\nu u$ and $\Delta u$. To this end, consider first the case $(u_0,u_1)=0$ and define a mapping
$$R_0:\, _0\mathbb{E}_1(J)\to \mathbb{E}_0(J)\times\, _0\mathbb{F}(J)$$
by $R_0 u:=[-\Delta u,\partial_\nu u]$. The linear problem
\begin{align}\label{eq:linwesterveltfull}
\begin{split}
c^{-2}u_{tt} -\Delta u- \beta \Delta u_t &= f,\quad\text{in }J\times\Omega,\\
     \partial_\nu u+ \beta \partial_\nu u_t+\alpha u_t&= g,\quad\text{in }J\times\partial\Omega,\\
      (u(0),u_t(0)) &= (0,0) ,\quad\text{in }\Omega,
     \end{split}
\end{align}
is then equivalent to the abstract equation
$$L_0u+R_0u=[f,g]$$
for $u\in\, _0\mathbb{E}_1(J)$ and some given functions $(f,g)\in \mathbb{E}_0(J)\times\, _0\mathbb{F}(J)$. Observe that
\begin{equation}\label{eq:opeq}
L_0 u+R_0u=L_0(I+L_0^{-1}R_0),
\end{equation}
since $L_0$ is invertible by Lemma \ref{lem:linwestervelt1}. Using the fact that the norm of $L_0^{-1}$ does not depend on $T\in (0,T_0]$ for some fixed $T_0>0$, it follows that there exists a constant $C=C(T_0)>0$ such that
$$\|L_0^{-1}R_0 u\|_{_0\mathbb{E}_1(J)}\le C\|R_0u\|_{\mathbb{E}_0(J)\times\, _0\mathbb{F}(J)}$$
for all $T\in (0,T_0]$. Furthermore we have
$$\|R_0 u\|_{\mathbb{E}_0(J)\times\, _0\mathbb{F}(J)}=\|\Delta u\|_{\mathbb{E}_0(J)}+\|\partial_\nu u\|_{_0\mathbb{F}(J)}.$$
We use the Sobolev embedding
$$_0W_p^1(0,T;W_p^2(\Omega))\hookrightarrow \{u\in C([0,T];W_p^2(\Omega)):u(0)=0\}$$
and the fact that the corresponding embedding constant $m>0$ is independent of $T>0$ to obtain the estimate
\begin{multline*}
\|\Delta u\|_{\mathbb{E}_0(J)}\le \|u\|_{L_p(0,T;W_p^2(\Omega))}\le T^{1/p}\|u\|_{C([0,T];W_p^2(\Omega))}\\
\le mT^{1/p}\|u\|_{_0W_p^1(0,T;W_p^2(\Omega))}\le mT^{1/p}\|u\|_{_0\mathbb{E}_1(J)}.
\end{multline*}
Furthermore, by trace theory we have
$$\|\partial_\nu u\|_{_0\mathbb{F}(J)}\le C\|u\|_{_0W_p^1(0,T;L_p(\Omega))\cap L_p(0,T;W_p^2(\Omega))},$$
where the constant $C>0$ is again independent of $T>0$. Making use of Sobolev embeddings again, we obtain as above a constant $C>0$ which does not depend on $T>0$ such that
$$\|\partial_\nu u\|_{_0\mathbb{F}(J)}\le CT^{1/p}\|u\|_{_0\mathbb{E}_1(J)}.$$
In summary we have shown that the estimate
$$\|L_0^{-1}R_0 u\|_{_0\mathbb{E}_1(J)}\le C\|R_0u\|_{\mathbb{E}_0(J)\times\, _0\mathbb{F}(J)}\le CT^{1/p}\|u\|_{_0\mathbb{E}_1(J)}$$
holds for all $u\in\, _0\mathbb{E}_1(J)$. Therefore, if $0<T<\min\{1/C^p,T_0\}$, a Neumann series argument yields that the operator $I+L_0^{-1}R_0:\, _0\mathbb{E}_1(J)\to\, _0\mathbb{E}_1(J)$ is invertible, hence, by \eqref{eq:opeq},
$$L_0+R_0:\, _0\mathbb{E}_1(J)\to\mathbb{E}_0(J)\times\, _0\mathbb{F}(J)$$
is invertible as well.

In a next step we take nontrivial initial values into account. To this end let $f\in \mathbb{E}_0(J)$, $g\in\mathbb{F}(J)$ and $(u_0,u_1)\in X_\gamma$ be given such that
$$\partial_\nu u_0+\beta\partial_\nu u_1+\alpha u_1=g(0)$$
on $\{t=0\}\times\partial\Omega$ if $p>3$. Extend $u_1\in W_p^{2-2/p}(\Omega)$ to some function $\tilde{u}_1\in W_p^{2-2/p}(\mathbb{R}^d)$, which is always possible by the assumption $\partial\Omega\in C^2$.
Then solve the full space problem
$$c^{-2}\tilde{w}_t-\beta\Delta \tilde{w}=0,\ \text{in}\ (0,T)\times\mathbb{R}^d,\quad \tilde{w}(0)=\tilde{u}_1\ \text{in}\ \mathbb{R}^d,$$
to obtain a unique solution
$$\tilde{w}\in W_p^1(0,T;L_p(\mathbb{R}^d))\cap L_p(0,T;W_p^2(\mathbb{R}^d)),$$
see e.g.\ \cite[Chapter II]{DHP03} or \cite[Chapter 6]{PS16}. This in turn implies that the restriction $w$ of $\tilde{w}$ to $\Omega$ satisfies
$$w\in W_p^1(0,T;L_p(\Omega))\cap L_p(0,T;W_p^2(\Omega))$$
and $w(0)=\tilde{u}_1|_{\Omega}=u_1$. Then we solve the abstract equation
\begin{equation}\label{eq:opeq2}
L_0 \hat{u}+R_0 \hat{u}=[\hat{f},\hat{g}],
\end{equation}
where $\hat{f}:=f+\Delta u_0+\Delta\int_0^t w(s) ds$ and
$$\hat{g}:=g-\partial_\nu u_0-\beta\partial_\nu w-\alpha w-\partial_\nu\int_0^t w(s) ds.$$
Since $\hat{f}\in\mathbb{E}_0(J)$ and $\hat{g}\in\, _0\mathbb{F}(J)$, this yields a unique solution $\hat{u}\in\, _0\mathbb{E}_1(J)$ of \eqref{eq:opeq2}. Defining
$$u:=u_0+\hat{u}+\int_0^tw(s)ds$$
it follows that $u\in\mathbb{E}_1(J)$ solves
\begin{align}\label{eq:linwesterveltfull2}
\begin{split}
c^{-2}u_{tt} -\Delta u- \beta \Delta u_t &= f,\quad\text{in }(0,T)\times\Omega,\\
     \partial_\nu u+ \beta \partial_\nu u_t+\alpha u_t&= g,\quad\text{in }(0,T)\times\partial\Omega,\\
      (u(0),u_t(0)) &= (u_0,u_1) ,\quad\text{in }\Omega,
     \end{split}
\end{align}
and the solution is unique by the considerations above. A successive application of this procedure yields a unique solution $u\in \mathbb{E}_1(J)$ on \emph{any finite} interval $(0,T)$. We have thus proven the following result.
\begin{The}\label{thm:linwestervelt1}
Let $p\in (1,\infty)$, $p\neq 3$, $\alpha\ge 0$, $J=(0,T)$, $T\in (0,\infty)$, $\Omega\subset\mathbb{R}^d$, $d\in\mathbb{N}$, be a bounded domain with boundary $\partial\Omega\in C^2$. Then there exists a unique solution
$$u\in W_p^2(J;L_p(\Omega))\cap W_p^1(J;W_p^2(\Omega))$$
of \eqref{eq:linwesterveltfull2} if and only if the data satisfy the following conditions.
\begin{enumerate}
\item $f\in L_p(J;L_p(\Omega))$;
\item $g\in W_p^{1/2-1/2p}(J;L_p(\partial\Omega))\cap L_p(J;W_p^{1-1/p}(\partial\Omega))$;
\item $(u_0,u_1)\in W_p^2(\Omega)\times W_p^{2-2/p}(\Omega)$;
\item $\partial_\nu u_0+\beta\partial_\nu u_1+\alpha u_1=g(0)$ if $p>3$.
\end{enumerate}
There exists a constant $C=C(T)>0$ such that the estimate
\begin{equation}\label{eq:MRest}
\|u\|_{\mathbb{E}_1(J)}\le C(\|f\|_{\mathbb{E}_0(J)}+\|g\|_{\mathbb{F}(J)}+\|(u_0,u_1)\|_{X_\gamma})
\end{equation}
is valid.
\end{The}
\begin{proof}
Necessity follows as in Lemma \ref{lem:linwestervelt1} and the estimate \eqref{eq:MRest} is a consequence of the open mapping theorem.
\end{proof}
\begin{Rem}
Note that \eqref{eq:linwesterveltfull2} does not have optimal regularity of type $L_p$ on $\mathbb{R}_+$. Indeed, for $u_0=c\in\mathbb{R}$ and $u_1=0$, the pair $(c,0)\in\mathbb{E}_1(J)$ is a solution of \eqref{eq:linwesterveltfull2} with $(f,g)=0$, but $(c,0)\notin \mathbb{E}_1(\mathbb{R}_+)$.
\end{Rem}

\section{Nonlinear well-posedness}\label{sec:NWP}

Let us start with the following regularity result. In order to keep things simple, we assume for a moment that $\gamma=\frac{1}{2}$ and $c=1$ in $\eqref{eq:westervelt}_2$.
\begin{Pro}\label{prop:regularity}
Let $p>d+1$, let $J=[0,T]$ for some $T\in (0,\infty)$ and assume that $\Omega\subset\mathbb{R}^d$ is a bounded domain with boundary $\partial\Omega\in C^2$. For $(t,x)\in J\times\Omega$ and
$$u\in \mathbb{V}(J):=\{v\in \mathbb{E}_1(J):\|v\|_{L_\infty(J;L_\infty(\Omega))}<1\},$$
define $F(u)(t,x):=\left(u_t(t,x)\sqrt{1-u(t,x)}\right)|_{\partial\Omega}$. Then
\begin{enumerate}
\item $F:\mathbb{V}(J)\to \mathbb{F}(J)$,
\item $F\in C^\infty(\mathbb{V}(J),\mathbb{F}(J))$,
\item $F'(v)\hat{u}=\left(\hat{u}_t\sqrt{1-v}+\frac{v_t\hat{u}_t}{2\sqrt{1-v}}\right)|_{\partial\Omega}$, for $v\in\mathbb{V}(J)$ and $\hat{u}\in\mathbb{E}_1(J)$,
\item $[u\mapsto (u^2)_{tt}]\in C^\infty(\mathbb{E}_1(J);\mathbb{E}_0(J))$.
\end{enumerate}
\end{Pro}
\begin{proof}
1. Note that $\mathbb{F}(J)$ is the trace space of the anisotropic space
$$H_p^{1/2}(J;L_p(\Omega))\cap L_p(J;W_p^{1}(\Omega)),$$
see e.g.\ \cite[Theorem 4.5]{MeyrSchn12}. Furthermore, by Sobolev embedding, it holds that
$$W_p^{1}(J;L_p(\Omega))\hookrightarrow H_p^{1/2}(J;L_p(\Omega))$$
for each $p>1$. Therefore it suffices to estimate $F(u)$ in the norm of the space
$$W_p^{1}(J;L_p(\Omega))\cap L_p(J;W_p^{1}(\Omega))=:\mathbb{X}(J).$$
Note that in case $p>d+1$, the space $\mathbb{X}(J)$ is a Banach algebra. Since $\|u_t\|_{\mathbb{X}(J)}\le \|u\|_{\mathbb{E}_1(J)}$, it remains to estimate $\sqrt{1-u}$ in the norm of $\mathbb{X}(J)$. It holds that
$$\partial_t \sqrt{1-u(t,x)}=-\frac{u_t(t,x)}{2\sqrt{1-u(t,x)}},$$
hence
$$\left\|\frac{u_t}{2\sqrt{1-u}}\right\|_{L_p(L_p)}\le C\frac{\|u\|_{\mathbb{E}_1(J)}}{\min_{(t,x)\in J\times\overline{\Omega}}\sqrt{1-u(t,x)}}<\infty,$$
since $\mathbb{E}_1(J)\hookrightarrow W_{p}^1(J;L_{p}(\Omega))$.
Now consider $\sqrt{1-u}$ in the norm of $L_p(J;W_p^1(\Omega))$. To this end it will be sufficient to estimate
$$\nabla \sqrt{1-u(t,x)}=-\frac{\nabla u(t,x)}{2\sqrt{1-u(t,x)}}$$
in $L_p(J;L_p(\Omega))$. Since $\mathbb{E}_1(J)\hookrightarrow L_p(J;W_p^1(\Omega))$, we obtain the same estimate as above, hence $\|F(u)\|_{\mathbb{X}(J)}<\infty$. This proves the first assertion.

2. We show that $F\in C^1$, the existence of the higher order derivatives follows inductively. Again, since $\mathbb{X}(J)$ is an algebra, it suffices to show that $[u\mapsto\sqrt{1-u}]\in C^1(\mathbb{V}(J);\mathbb{X}(J))$. Fix $u\in \mathbb{V}(J)$ and let $\|h\|_{\mathbb{E}_1(J)}\le\delta$ with $\delta>0$ being sufficiently small such that $u+h\in \mathbb{V}(J)$. This is possible, since $\mathbb{V}(J)$ is open in $\mathbb{E}_1(J)$. By the fundamental theorem of calculus, we obtain the identity
\begin{multline*}
G(u(t,x)+h(t,x))-G(u(t,x))-G'(u(t,x))h(t,x)=\\
=\int_0^1\int_0^1 G''(u(t,x)+s\tau h(t,x))dsd\tau h(t,x)^2,
\end{multline*}
where $G(r):=\sqrt{1-r}$ and $r<1$. It is easy to see that $\|G''(u+s\tau h)\|_{\mathbb{X}(J)}$ is uniformly bounded with respect to $s,\tau\in[0,1]$ and $\|h\|_{\mathbb{E}_1(J)}\le\delta$. Therefore, the algebra property of $\mathbb{X}(J)$ and the fact that $\mathbb{E}_1(J)\hookrightarrow \mathbb{X}(J)$ yields that $[u\mapsto G(u)]$ is Frechet differentiable with derivative
$$G'(v)\hat{u}=-\frac{1}{2\sqrt{1-v}}\hat{u}.$$
valid for all $v\in\mathbb{V}(J)$ and $\hat{u}\in \mathbb{E}_1(J)$.
The continuity of the derivative follows in a very similar way, we skip the details.

3. The proof of this assertion follows directly from the proof of the second assertion and the product rule.

4. This statement has been proven in \cite[Section 3]{MeWi11} and \cite[Proof of Lemma~6]{MeWi13}.
\end{proof}
\begin{proof}[Proof of Theorem \ref{thm:mainthm}]
We will solve \eqref{eq:westervelt} by means of the implicit function theorem. To this end, let $T>0$ be fixed. Note that in case $p>3$ we have to take into account the nonlinear compatibility condition
\begin{equation}\label{eq:compcondnonlin}
\partial_\nu u_0+\beta\partial_\nu u_1+u_1\sqrt{c^{-2}-2\gamma u_0}=0
\end{equation}
between the initial vaules and the boundary condition on $\partial\Omega$. To this end, let $g=0$ if $p<3$ and
$$g(t):=e^{\Delta_{\partial\Omega}t}(\partial_\nu u_0+\beta\partial_\nu u_1),\ t\ge 0,$$
if $p>3$. Here $\Delta_{\partial\Omega}$ denotes the Laplace-Beltrami operator on $\partial\Omega$. It is well known that if $(\partial_\nu u_0+\beta\partial_\nu u_1)\in W_p^{1-3/p}(\partial\Omega)$ then
$$g\in W_p^{1/2-1/2p}(J;L_p(\partial\Omega))\cap L_p(J;W_p^{1-1/p}(\partial\Omega)),$$
see e.g.\ \cite[Proposition 3.4.3]{PS16}.
By Theorem \ref{thm:linwestervelt1} with $\alpha=0$ there exists a unique solution $u^*=u^*(u_0,u_1)\in \mathbb{E}_1(J)$ of
\begin{align}
\begin{split}\label{eq:linaux1}
c^{-2}u_{tt} -\Delta u- \beta \Delta u_t &= 0,\quad\text{in }(0,T)\times\Omega,\\
     \partial_\nu u+ \beta \partial_\nu u_t&= g,\quad\text{in }(0,T)\times\partial\Omega,\\
      (u(0),u_t(0)) &= (u_0,u_1) ,\quad\text{in}\ \Omega.
     \end{split}
\end{align}
Choose $\delta>0$ sufficiently small such that if $\|(u_0,u_1)\|_{X_\gamma}<\delta$, then
$$\max_{t\in[0,T]}\|u^*(t)\|_\infty<\frac{1}{4\gamma c^2}.$$
This is possible, since
$$\|u^*\|_{\mathbb{E}_1(J)}\le C(\|g\|_{\mathbb{F}(J)}+\|(u_0,u_1)\|_{X_\gamma})\le \tilde{C}\|(u_0,u_1)\|_{X_\gamma},$$
and $\mathbb{E}_1(J)\hookrightarrow C([0,T];C(\overline{\Omega}))$ provided that $p>d/2$. Let
\begin{equation}\label{eq:defbbW}
_0\mathbb{W}(J):=\left\{u\in\, _0\mathbb{E}_1(J):\max_{t\in[0,T]}\|u(t)\|_\infty<\frac{1}{4\gamma c^2}\right\}.
\end{equation}
Then $_0\mathbb{W}(J)$ is an open subset of $_0\mathbb{E}_1(J)$ provided that $p>d/2$. Next, we define a nonlinear mapping $H:\, _0\mathbb{W}(J)\times B_{X_\gamma}((0,0),\delta)\to \mathbb{E}_0(J)\times\, _0\mathbb{F}(J)$ by
$$H(u,(u_0,u_1)):=
\begin{bmatrix}
c^{-2}(u+u^*)_{tt} -\Delta (u+u^*)- \beta \Delta (u+u^*)_t -\gamma [(u+u^*)^2]_{tt}\\
 \partial_\nu u+ \beta \partial_\nu u_t+(u_t+u^*_t)\sqrt{c^{-2}-2\gamma (u+u^*)}-h.
\end{bmatrix},
$$
where $h=0$ if $p<3$ and
$$h(t)=e^{\Delta_{\partial\Omega}t}\left([u_1\sqrt{c^{-2}-2\gamma u_0}]|_{\partial\Omega}\right),\ t\ge 0,$$
if $p>3$. Then $h\in \mathbb{F}(J)$, since $[\sqrt{c^{-2}-2\gamma u_0}u_1]|_{\partial\Omega}\in W_p^{1-3/p}(\partial\Omega)$.

Note that $H(0,(0,0))=0$, since $g=0$ if $(u_0,u_1)=0$ and then $u^*=0$ is the unique solution of \eqref{eq:linaux1}. Furthermore
$$\max_{t\in[0,T]}\|u(t)+u^*(t)\|_\infty<\frac{1}{4\gamma c^2}+\frac{1}{4\gamma c^2}=\frac{1}{2\gamma c^2}$$
for all $(u,(u_0,u_1))\in\, _0 \mathbb{W}(J)\times B_{X_\gamma}((0,0),\delta)$. Since the linear mapping $[(u_0,u_1)\mapsto u^*(u_0,u_1)]$ from $X_\gamma$ to $\mathbb{E}_1(J)$ is smooth, it follows from Proposition \ref{prop:regularity} that $H\in C^\infty$ and
$$D_uH(0,(0,0))\hat{u}=
\begin{bmatrix}
c^{-2}\hat{u}_{tt} -\Delta \hat{u}- \beta \Delta \hat{u}_t \\
\partial_\nu \hat{u}+ \beta \partial_\nu \hat{u}_t+c^{-1}\hat{u}_t
\end{bmatrix}.
$$
By Theorem \ref{thm:linwestervelt1} with $\alpha=c^{-1}$, the operator $D_u H(0,(0,0))$ is invertible, hence, by the implicit function theorem, there exists a ball $B_{X_\gamma}((0,0),r)$, $0<r<\delta$ and a unique function $\psi\in C^\infty(B_{X_\gamma}((0,0),r);\, _0\mathbb{W}(J))$ such that $H(\psi(u_0,u_1),(u_0,u_1))=0$ for all $(u_0,u_1)\in B_{X_\gamma}((0,0),r)$ and $\psi(0,0)=0$. Then
$$u:=u(u_0,u_1):=\psi(u_0,u_1)+u^*(u_0,u_1)$$
is the unique solution of \eqref{eq:westervelt} provided that $(u_0,u_1)$ satisfy \eqref{eq:compcondnonlin} in case $p>3$.

Since $\psi$ as well as $u^*$ are continuous in $(u_0,u_1)$, the continuity of the data-to-solution map $[(u_0,u_1)\mapsto u(u_0,u_1)]$ follows readily.
\end{proof}
\begin{Rem}\label{rem:remarkonp}
The statements of Proposition \ref{prop:regularity} and Theorem \ref{thm:mainthm} remain true if one replaces the assumption $p>d+1$ by the weaker condition $p>\max\{\frac{d}{2},\frac{d}{4}+1\}$. Since in case $p>\max\{\frac{d}{2},\frac{d}{4}+1\}$ one cannot work with the algebra property of the space
$$W_p^{1}(J;L_p(\Omega))\cap L_p(J;W_p^{1}(\Omega)),$$
the proof of Proposition \ref{prop:regularity} requires more subtle estimates using H\"older's inequality and various Sobolev embeddings (see also \cite{MeWi11,MeWi13}). For the sake of simplicity we assumed the slightly stronger assumption $p>d+1$.
\end{Rem}

\section{Higher regularity}

Let $u_*\in\mathbb{E}_1(J)$ be the unique solution to \eqref{eq:westervelt} which exists thanks to Theorem \ref{thm:mainthm}. Let $\varepsilon\in (0,1)$ be fixed but as small as we please and let $J_\varepsilon=[0,T/(1+\varepsilon)]$. For $t\in J_\varepsilon$ and $\lambda\in (1-\varepsilon,1+\varepsilon)$ we define $u_\lambda(t):=u_*(\lambda t)$. Then $u_\lambda\in\mathbb{E}_1(J_\varepsilon)$ and $u_\lambda$ is a solution of the problem
\begin{align}\label{eq:westhighreg1}
\begin{split}
c^{-2}\partial_t^2u_\lambda - \lambda^2\Delta u_\lambda - \lambda\beta \Delta \partial_tu_\lambda &= \gamma (u_\lambda^2)_{tt},\quad\text{in }J_\varepsilon\times\Omega,\\
      \partial_\nu (\lambda u_\lambda+\beta \partial_tu_\lambda)+\partial_tu_\lambda\sqrt{c^{-2}-2\gamma u_\lambda}&= 0,\quad\text{in }J_\varepsilon\times\partial\Omega,\\
      (u_\lambda(0),\partial_tu_\lambda(0)) &= (u_0,\lambda u_1) ,\quad\text{in }\Omega.
     \end{split}
\end{align}
For given $\lambda\in (1-\varepsilon,1+\varepsilon)$ we solve the problem
\begin{align}
\begin{split}\label{eq:highregaux1}
c^{-2}v_{tt} -\Delta v- \beta \Delta v_t &= 0,\quad\text{in }(0,T)\times\Omega,\\
     \partial_\nu v+ \beta \partial_\nu v_t&= g(\lambda),\quad\text{in }(0,T)\times\partial\Omega,\\
      (v(0),v_t(0)) &= (u_0,\lambda u_1) ,\quad\text{in}\ \Omega,
     \end{split}
\end{align}
where $g(\lambda)=0$ if $p<3$ and
$$g(\lambda)=e^{\Delta_{\partial\Omega}t}[\partial_\nu u_0+\lambda\beta\partial_\nu u_1],$$
if $p>3$. By Theorem \ref{thm:linwestervelt1} this yields a unique solution $v(\lambda)\in \mathbb{E}_1(J_\varepsilon)$. We note on the go that the mapping $[\lambda\mapsto v(\lambda)]$ from $(1-\varepsilon,1+\varepsilon)$ to $\mathbb{E}_1(J_\varepsilon)$ is $C^\infty$, since the parameter $\lambda$ appears only polynomially in the linear problem \eqref{eq:highregaux1}.

Choose $\varepsilon>0$ and $\|(u_0,u_1)\|_{X_\gamma}$ sufficiently small such that
$$\sup_{\lambda\in (1-\varepsilon,1+\varepsilon)}\left(\max_{t\in J_\varepsilon}\|u_*(\lambda t)\|_\infty+\max_{t\in J_\varepsilon}\|[v(\lambda)](t)\|_\infty\right)<\frac{1}{4\gamma c^2}.$$
This is always possible by estimate \eqref{eq:MRest} and by the continuous dependence of $u_*$ on $(u_0,u_1)$ (which is uniform w.r.t.\ $T$). Note that $u_*=0$ if $(u_0,u_1)=0$, by uniqueness of the solution of \eqref{eq:westervelt}.

Let $_0\mathbb{W}(J_\varepsilon)$ be as in \eqref{eq:defbbW} with $J$ being replaced by $J_\varepsilon$ and define a mapping $H:(1-\varepsilon,1+\varepsilon)\times\, _0\mathbb{W}(J_\varepsilon)\to\mathbb{E}_0(J_\varepsilon)\times\, _0\mathbb{F}(J_\varepsilon)$ by
$$H(\lambda,u):=
\begin{bmatrix}
c^{-2}(u+v(\lambda))_{tt} -\lambda^2\Delta (u+v(\lambda))- \lambda\beta \Delta (u+v(\lambda))_t -\gamma [(u+v(\lambda))^2]_{tt}\\
 \lambda\partial_\nu (u+v(\lambda))+ \beta \partial_\nu (u+v(\lambda))_t+(u+v(\lambda))_t\sqrt{c^{-2}-2\gamma (u+v(\lambda))}.
\end{bmatrix}.
$$
Since $[\lambda\mapsto v(\lambda)]$ is $C^\infty$ it follows that $H\in C^\infty$ in $(\lambda,u)$ as well. Furthermore it holds that $H(1,u_*-v(1))=0$ and
$$D_uH(1,u_*-v(1))\hat{u}=
\begin{bmatrix}
c^{-2}\hat{u}_{tt} -\Delta \hat{u}- \beta \Delta \hat{u}_t-2\gamma(u_*\hat{u})_{tt}\\
\partial_\nu \hat{u}+ \beta \partial_\nu \hat{u}_t+\hat{u}_t\sqrt{c^{-2}-2\gamma u_*}-\gamma\frac{\hat{u}\partial_t u_*}{\sqrt{c^{-2}-2\gamma u_*}}
\end{bmatrix},
$$
by Proposition \ref{prop:regularity}. A Neumann series argument implies that
$$D_uH(1,u_*-v(1)):\, _0\mathbb{E}_1(J_\varepsilon)\to\mathbb{E}_0(J_\varepsilon)\times\, _0\mathbb{F}(J_\varepsilon)$$
is invertible, provided that the norm $\|u_*\|_{\mathbb{E}_1(J_\varepsilon)}$ is sufficiently small. This follows readily by decreasing $\|(u_0,u_1)\|_{X_\gamma}$ if necessary.

By the implicit function theorem there exists $r\in (0,\varepsilon)$ and a unique mapping $\phi\in C^\infty((1-r,1+r);\, _0\mathbb{W}(J_\varepsilon))$ such that $H(\lambda,\phi(\lambda))=0$ for all $\lambda\in (1-r,1+r)$ and $\phi(1)=u_*-v(1)$. From the uniqueness it follows that $u_\lambda=\phi(\lambda)+v(\lambda)$, hence
$$[\lambda\mapsto u_\lambda]\in C^\infty((1-r,1+r);\mathbb{E}_1(J_\varepsilon)).$$
Since $\partial_\lambda u_\lambda (t)=t\dot{u}_*(\lambda t)$ one computes inductively that
$$[t\mapsto t^k u_*^{(k)}(t)]\in \mathbb{E}_1(J)$$
for each $k\in\mathbb{N}_0$. Note that one may pass to the limit $\varepsilon\to 0$, since one evaluates the above derivatives at $\lambda=1$. In particular, this yields
$$u_*\in W_p^{k+2}(\tau,T;L_p(\Omega))\cap W_p^{k+1}(\tau,T;W_p^2(\Omega)),$$
for all $k\in\mathbb{N}$ and each $\tau\in (0,T)$. Moreover, by Sobolev embedding, it holds that
$$u_*\in C^\infty(0,T;W_p^2(\Omega)).$$
We have thus proven the following result.
\begin{The}\label{thm:HR}
Let the conditions of Theorem \ref{thm:mainthm} be satisfied. Then for each $T\in(0,\infty)$ there exists $\delta>0$ such that for all $u_0\in W_p^2(\Omega)$ and $u_1\in W_p^{2-2/p}(\Omega)$ satisfying
the estimate
$$\|u_0\|_{W_p^2(\Omega)}+\|u_1\|_{W_p^{2-2/p}(\Omega)}\le\delta,$$  and the compatibility condition
\begin{equation}
\partial_\nu (u_0+\beta u_1)+u_1\sqrt{c^{-2}-2\gamma u_0}= 0,\quad\text{on }\partial\Omega
\end{equation}
if $p>3$,
the unique solution
$$u\in W_p^2(J;L_p(\Omega))\cap W_p^1(J;W_p^2(\Omega))$$
of \eqref{eq:westervelt} satisfies
$$u\in W_p^{k+2}(\tau,T;L_p(\Omega))\cap W_p^{k+1}(\tau,T;W_p^2(\Omega)),$$
for all $k\in\mathbb{N}$ and each $\tau\in (0,T)$. In particular it holds that
$$u\in C^\infty(0,T;W_p^2(\Omega)).$$
\end{The}

\section{Long-Time Behaviour}

\noindent
In this section we assume that $p>d+1$. Note that as long as $|u(t,x)|<\frac{1}{2\gamma c^2}$ for all $(t,x)\in (0,T)\times\Omega$ we can rewrite \eqref{eq:westervelt} as the first order system
\begin{equation}\label{eq:1storder1}
\partial_t w={A}(w)w+F(w),
\end{equation}
subject to the nonlinear boundary condition
\begin{equation}\label{eq:1storder2}
{B}(w)=0,
\end{equation}
where $w=(u,v)=(u,u_t)$,
$${A}(w):=
\begin{bmatrix}
0 & I\\ \frac{1}{c^{-2}-2\gamma u}\Delta  & \frac{\beta}{c^{-2}-2\gamma u}\Delta
\end{bmatrix},\quad F(w)=
\begin{bmatrix}
0 \\ \frac{2\gamma v^2}{c^{-2}-2\gamma u}
\end{bmatrix},$$
and
$${B}(w):=\partial_\nu u+\beta\partial_\nu v+v\sqrt{c^{-2}-2\gamma u}.$$
As in Section \ref{sec:NWP} one can show that the mapping $[w\mapsto (A(w),F(w),B(w))]$ is smooth, as long as the first component $u$ of $w$ is bounded away from the critical value $\frac{1}{2\gamma c^2}$ and provided $p>d+1$.

Note that the set of equilibria $\mathcal{E}$ of this first order system (or equivalently \eqref{eq:westervelt}) is given by
$$\mathcal{E}=\left\{(r,0):r\in\mathbb{R}\ \text{and}\ |r|<\frac{1}{2\gamma c^2}\right\}.$$
To study the stability properties of such an equilibrium, we consider the full linearization of \eqref{eq:1storder1} and \eqref{eq:1storder2} in $(r,0)\in\mathcal{E}$. This yields a linear operator $A_0$ defined by
$$A_0w=A_0(u,v)=
\begin{bmatrix}
v \\ \frac{1}{c^{-2}-2\gamma r}\Delta u+\frac{\beta}{c^{-2}-2\gamma r}\Delta v
\end{bmatrix}
$$
in the Banach space $X_0:=W_p^2(\Omega)\times L_p(\Omega)$, equipped with the domain $X_1:=D(A_0)$ given by
$$
X_1=\left\{w=(u,v)\in W_p^2(\Omega)\times W_p^2(\Omega):\partial_\nu u+\beta\partial_\nu v+v\sqrt{c^{-2}-2\gamma r}=0\ \text{on}\ \partial\Omega\right\}.
$$
By Theorem \ref{thm:linwestervelt1} the operator $A_0$ has the property of maximal $L_p$-regularity on each bounded interval $[0,T]$. Therefore $A_0$ is the generator of an analytic $C_0$-semigroup in $X_0$, see e.g.\ \cite[Proposition 1.2]{Pru02b} or \cite[Theorem 2.2]{Dor93}.

In what follows, we will investigate the spectrum $\sigma(A_0)$ of $A_0$. Note that $D(A_0)$ is not compactly embedded into $X_0$, hence we cannot work with a compact resolvent of $A_0$. In a first step we show that the inclusion
$$\sigma_{app}(A_0)\subset\mathbb{C}_-\cup\{0\},$$
holds for the approximate point spectrum $\sigma_{app}(A_0)$ of $A_0$. Clearly, $\lambda=0$ is an eigenvalue of $A_0$ with the corresponding eigenspace
$$\{(u,v)\in X_1:u\ \text{is constant}\ \text{and}\ v=0\}.$$
Let $0\neq\lambda\in\sigma_{app}(A_0)$. Then we find a sequence $(w_n)_n=(u_n,v_n)_n\subset X_1$ with $\|(u_n,v_n)\|_{X_0}=1$ such that
$$\lambda w_n-A_0w_n\to 0$$
in $X_0$ as $n\to\infty$, see e.g.\ \cite[Lemma IV.1.9]{EN00}. Setting $c_r:=\sqrt{c^{-2}-2\gamma r}>0$,
\begin{equation}\label{eq:appspec1}
\lambda u_n-v_n=:g_n
\end{equation}
and
\begin{equation}\label{eq:appspec2}
\lambda v_n-c_r^{-2}(\Delta u_n+\beta\Delta v_n)=:h_n
\end{equation}
this yields $g_n\to 0$ in $W_p^2(\Omega)$ and $h_n\to 0$ in $L_p(\Omega)$. We test the second equation by $\overline{v_n}$ and integrate by parts to the result
$$\lambda \|v_n\|_{L_2(\Omega)}^2+c_r^{-2}\beta\|\nabla v_n\|_{L_2(\Omega)^d}^2
+c_r^{-1}\|v_n\|_{L_2(\partial\Omega)}^2+c_r^{-2}(\nabla u_n|\nabla v_n)_{L_2(\Omega)}=(h_n|v_n)_{L_2(\Omega)}.$$
Since $u_n=\frac{1}{\lambda}(v_n+g_n)$ (by \eqref{eq:appspec1}) we obtain (after taking real parts)
\begin{multline*}
\operatorname{Re}\lambda \|v_n\|_{L_2(\Omega)}^2+c_r^{-2}\left(\frac{\operatorname{Re}\lambda}{|\lambda|^2}+\beta\right)\|\nabla v_n\|_{L_2(\Omega)^d}^2
+c_r^{-1}\|v_n\|_{L_2(\partial\Omega)}^2=\\
=\operatorname{Re}(h_n|v_n)_{L_2(\Omega)}-c_r^{-2}
\operatorname{Re}\left[\frac{\bar{\lambda}}{|\lambda|^2}(\nabla g_n|\nabla v_n)_{L_2(\Omega)}\right].
\end{multline*}
Applying the inequalities of Cauchy-Schwarz and Young to both terms on the right hand side yields
$$\operatorname{Re}(h_n|v_n)_{L_2(\Omega)}\le\varepsilon\|v_n\|_{L_2(\Omega)}^2+
C(\varepsilon)\|h_n\|_{L_2(\Omega)}^2$$
and
$$\operatorname{Re}\left[\frac{\bar{\lambda}}{|\lambda|^2}(\nabla g_n|\nabla v_n)_{L_2(\Omega)}\right]\le\frac{1}{|\lambda|}\left(\varepsilon\|\nabla v_n\|_{L_2(\Omega)^d}^2+
C(\varepsilon)\|\nabla g_n\|_{L_2(\Omega)^d}^2\right),$$
for an arbitrarily small $\varepsilon>0$ and some constant $C(\varepsilon)>0$.
Assume that $\operatorname{Re}\lambda\ge 0$ and $\lambda\neq 0$. Choosing $\varepsilon>0$ small enough and making use of the Poincar\'{e}-type inequality
$$\|v_n\|_{L_2(\Omega)}^2\le C\left(\|\nabla v_n\|_{L_2(\Omega)^d}^2+\|v_n\|_{L_2(\partial\Omega)}^2\right)$$
for some constant $C>0$ (being independent of $n$), we obtain an estimate of the form
$$\|v_n\|_{W_2^1(\Omega)}^2\le C(\lambda,d,r,\beta)\left(\|h_n\|_{L_2(\Omega)}^2+\|\nabla g_n\|_{L_2(\Omega)^d}^2\right).$$
Since $p> d+1\ge 2$ we may pass to the limit $n\to\infty$ which yields $\|v_n\|_{W_2^1(\Omega)}\to 0$, hence, by Sobolev embeddings, $\|v_n\|_{L_{q_0}(\Omega)}\to 0$ as $n\to\infty$, where $q_0=\frac{2d}{d-2}$ if $d\ge 3$ and $q_0=p$ if $d\le 2$. If $d\ge 3$, we distinguish two cases:
\begin{enumerate}
\item $q_0\ge p$: Then $v_n\to 0$ in $L_p(\Omega)$ and $\|v_n\|_{W_p^2(\Omega)}\le M$ for all $n\in\mathbb{N}$ and some constant $M>0$, by \eqref{eq:appspec1} and the assumption $\|(u_n,v_n)\|_{X_0}=1$. Interpolation theory yields $v_n\to 0$ in $W_p^{2s}(\Omega)$ for any $s\in (0,1)$, hence also $v_n|_{\partial\Omega}\to 0$ in $W_p^{1-1/p}(\partial\Omega)$, provided $2s\ge 1$. From now on we fix such an $s\in [1/2,1)$. Replacing $u_n$ in \eqref{eq:appspec2} and in the boundary condition
$$\partial_\nu u_n+\beta\partial_\nu v_n+c_rv_n=0$$
by \eqref{eq:appspec1}, we obtain the following linear elliptic problem for $v_n$:
\begin{align}
\begin{split}
\omega v_n-\Delta v_n&=\frac{\lambda}{1+\beta\lambda}\left(c_r^2 h_n+\frac{1}{\lambda}\Delta g_n-c_r^2\lambda v_n\right)+\omega v_n,\ x\in\Omega,\\
\partial_\nu v_n&=-\frac{\lambda}{1+\beta\lambda}\left(c_r v_n+\frac{1}{\lambda}\partial_\nu g_n\right),\ x\in\partial\Omega.
\end{split}
\end{align}
Here the number $\omega>0$ is arbitrary but fixed. Elliptic regularity theory for this inhomogeneous Neumann boundary value problem implies that $\|v_n\|_{W_p^2(\Omega)}\to 0$ as $n\to \infty$. Together with \eqref{eq:appspec1} this yields $u_n\to 0$ in $W_p^2(\Omega)$, which contradicts the fact $\|(u_n,v_n)\|_{X_0}=1$.
\item $q_0<p$: In this case we obtain from \eqref{eq:appspec1}, from the assumption $\|(u_n,v_n)\|_{X_0}=1$ and interpolation, that $v_n\to 0$ in $W_{q_0}^{2s}(\Omega)$ for any $s\in (0,1)$.
If the Sobolev index of the space $W_{q_0}^{2s}(\Omega)$ satisfies $2s-d/q_0\ge 0$, then
$$W_{q_0}^{2s}(\Omega)\hookrightarrow L_p(\Omega).$$
Therefore $v_n\to 0$ in $L_p(\Omega)$, hence we may follow the lines of case 1 to obtain a contradiction. If on the contrary $2s-d/q_0<0$, then we use the embedding
$$W_{q_0}^{2s}(\Omega)\hookrightarrow L_{q_1}(\Omega),$$
where
$$\frac{1}{q_1}=\frac{1}{q_0}-\frac{2s}{d}\in (0,1).$$
This yields $v_n\to 0$ in $L_{q_1}(\Omega)$ as $n\to \infty$. If $q_1$ can be chosen greater or equal to $p$, then we may follow the lines of case 1 above to obtain a contradiction. In case that $q_1< p$, we obtain (by Sobolev embedding and interpolation) that $v_n\to 0$ in $W_{q_1}^{2s}(\Omega)$ for each $s\in (0,1)$. In case $2s-d/q_1\ge 0$ we obtain as above $v_n\to 0$ in $L_p(\Omega)$, while for $2s-d/q_1<0$ we define
$$\frac{1}{q_2}=\frac{1}{q_1}-\frac{2s}{d}=\frac{1}{q_0}-2\frac{2s}{d}.$$
We may now iterate this procedure. Assume that for each $k\in\mathbb{N}$ it holds that $q_k<p$ and
$$\frac{1}{q_k}=\frac{1}{q_{k-1}}-\frac{2s}{d}\in (0,1).$$
This implies
$$\frac{1}{q_k}=\frac{1}{q_0}-k\frac{2s}{d},$$
hence $1/q_k<0$, if $k>\frac{d}{2sq_0}$, a contradiction. Therefore, there exists $k\in \mathbb{N}$ such that $q_k\ge p$ or $2s-d/q_{k-1}\ge 0$, which allows us to follow the lines of case 1.
\end{enumerate}
We have shown that if $\lambda\in\sigma_{app}(A_0)$, then $\lambda=0$ or $\operatorname{Re}\lambda<0$. Now it is well-known that for the topological boundary $\partial\sigma(A_0)$ of the spectrum of $A_0$ it holds that
$$\partial\sigma(A_0)\subset\sigma_{app}(A_0),$$
see e.g.\ \cite[Proposition IV.1.10]{EN00}. Assume that there exists $\lambda\in\sigma(A_0)$ with $\operatorname{Re}\lambda>0$. Then it follows that $\partial\sigma(A_0)\cap\mathbb{C}_+\neq\emptyset$, since $A_0$ generates an analytic $C_0$-semigroup in $X_0$. But this is impossible, since $\sigma_{app}(A_0)\subset\mathbb{C}_-\cup\{0\}$ and therefore it holds that
$$\sigma(A_0)\subset\overline{\mathbb{C}_-}.$$
Suppose that there exists $\lambda\in\sigma(A_0)$ such that $\operatorname{Re}\lambda=0$ and $\lambda\neq 0$. Then $\lambda\in\partial\sigma(A_0)\subset\sigma_{app}(A_0)$ which is a contradiction. This shows that
$$\sigma(A_0)\subset\mathbb{C}_-\cup\{0\}.$$
We claim that $\lambda=0\in\sigma(A_0)$ is semi-simple, i.e.\ $R(A_0)$ is closed in $X_0$ and
$$X_0=N(A_0)\oplus R(A_0).$$
Let $f=(g,h)\in R(A_0)\subset X_0$. Then there exists $w=(u,v)\in D(A_0)$ such that $A_0 w=f$ or equivalently $v=g$, $\Delta u+\beta\Delta v=c_r^2h$ in $\Omega$ and $\partial_\nu(u+\beta v)+c_rv=0$ on $\partial\Omega$; recall that $c_r=\sqrt{c^{-2}-2\gamma r}>0$. Integrating the second equation w.r.t\ $x\in\Omega$ yields
$$c_r^2\int_\Omega h dx=-c_r\int_{\partial\Omega} vd\sigma=-c_r\int_{\partial\Omega}g d\sigma,$$
($d\sigma$ denoting the surface measure on $\partial\Omega$), hence
$$c_r\int_\Omega h dx+\int_{\partial\Omega}g d\sigma=0.$$
Now we assume that
$$f=(g,h)\in\left\{(g,h)\in X_0:c_r\int_\Omega h dx+\int_{\partial\Omega}g d\sigma=0\right\}$$
is given. Define $v:= g\in W_p^2(\Omega)$ and consider the elliptic problem
$$\begin{cases}
\Delta u=c_r^2h-\beta\Delta g,&\ x\in\Omega,\\
\partial_\nu u=-(\beta \partial_\nu g+c_r g),&\ x\in\partial\Omega.
\end{cases}$$
Since
$$\int_\Omega(c_r^2h-\beta\Delta g)dx=-\int_{\partial\Omega} (\beta \partial_\nu g+c_r g)d\sigma$$
it is well known that there exists a solution $u\in W_p^2(\Omega)$ of this elliptic problem (being unique up to an additive constant). It follows that $w:=(u,v)\in D(A_0)$ and $A_0 w=f$, hence
$$R(A_0)=\left\{(g,h)\in X_0:c_r\int_\Omega h dx+\int_{\partial\Omega}g d\sigma=0\right\}.$$
This in turn implies that $R(A_0)$ is closed in $X_0$. Let $f=(g,h)\in X_0$ be given. Then we may write $(g,h)=(k,0)+(g-k,h)$, where
\begin{equation}\label{eq:defofk}
k:=\frac{1}{|\partial\Omega|}\left(c_r\int_\Omega h dx+\int_{\partial\Omega}g d\sigma\right).
\end{equation}
With this choice it follows that $(g-k,h)\in R(A_0)$ and (of course) $(k,0)\in N(A_0)=\operatorname{span}\{1\}\times\{0\}$. Define a mapping $P:X_0\to X_0$ by $P(g,h):=(k,0)$ where $k$ is given by \eqref{eq:defofk}. It is easily seen that $P$ is a continuous projection with $N(P)=R(A_0)$ and $R(P)=N(A_0)$. Therefore it holds that $X_0=N(A_0)\oplus R(A_0)$, hence $\lambda=0$ is semi-simple.

We are now in a position to follow the lines of the proof of \cite[Theorem 3.1]{PSZ09} to obtain the following result on the qualitative behaviour of the solution of \eqref{eq:westervelt} in a neighbourhood of an equilibrium.
\begin{The}\label{thm:Stab}
Let the conditions of Theorem \ref{thm:mainthm} be satisfied and let $(r_*,0)\in \mathcal{E}$ be an equilibrium.

Then $(r_*,0)$ is stable in $X_\gamma=W_p^2(\Omega)\times W_p^{2-2/p}(\Omega)$ and there exists $\delta>0$ such that the solution $u(t)$ of \eqref{eq:westervelt} with initial value $(u_0,u_1)\in X_\gamma$, satisfying
$$\|u_0-r_*\|_{W_p^2(\Omega)}+\|u_1\|_{W_p^{2-2/p}(\Omega)}\le\delta$$
and the compatibility condition
$$\partial_\nu (u_0+\beta u_1)+u_1\sqrt{c^{-2}-2\gamma u_0}= 0,\quad\text{on }\partial\Omega\quad (\text{if}\ p>3),$$
exists on $\mathbb{R}_+$ and $(u(t),u_t(t))$ converges exponentially fast in $X_\gamma$ to some $(r_\infty,0)\in\mathcal{E}$ as $t\to\infty$.
\end{The}

\bibliographystyle{abbrv}
\bibliography{MeWi11_Lit}

\end{document}